\documentclass{article}

\usepackage{amsmath,amssymb,amsthm,graphicx,proof,color}
\usepackage[UKenglish]{babel}

\usepackage{hyperref}
\usepackage[all]{xy}
\usepackage{times}

\hypersetup{
   colorlinks,%
   citecolor=blue,%
   filecolor=black,%
   linkcolor=red,%
   urlcolor=black
 }

\usepackage[all]{xy}
\usepackage{tikz}
\usetikzlibrary{trees}
\usepgflibrary{arrows}
\usetikzlibrary{positioning,shapes}
\usetikzlibrary{patterns}

\newcommand{\M}{\mathcal{M}}

\newcommand{\BP}{\ensuremath{\textbf{P}}}

\newcommand{\lr}[1]{\langle #1 \rangle}
\newcommand{\lra}{\leftrightarrow}

\renewcommand{\phi}{\varphi}

\newtheorem{theorem}{Theorem}%��Щ��һЩpackage�ﶼ�ж���,�õ�ʱ����ϵͳ��ʾ.
\newtheorem{lemma}[theorem]{Lemma}
\newtheorem{definition}[theorem]{Definition}
\newtheorem{proposition}[theorem]{Proposition}

\newcommand{\weg}[1]{}

\title{A sequence of neighborhood contingency logics}

\author{Jie Fan\\
\small School of Philosophy, Beijing Normal University  \\
\small \texttt{fanjie@bnu.edu.cn}}
\date{Submitted to some journal on Nov. 8, 2017}
\begin{document}
\maketitle

\begin{abstract}
This note proposes various axiomatizations of contingency logic under neighborhood semantics. In particular, by defining a suitable canonical neighborhood function, we give sound and complete axiomatizations of monotone contingency logic and regular contingency logic, thereby answering two open questions raised by Bakhtiari, van Ditmarsch, and Hansen. The canonical function is inspired by a function proposed by Kuhn in~1995. We show that Kuhn's function is actually equal to a related function originally given by Humberstone.
\end{abstract}


\section{Introduction}

Compared to normal modal logics, non-normal modal logics usually have many disadvantages, such as weak expressivity, weak frame definability, which leads to non-triviality of axiomatizations. Contingency logic is such a logic~\cite{MR66,Cresswell88,Humberstone95,DBLP:journals/ndjfl/Kuhn95,DBLP:journals/ndjfl/Zolin99,Fanetal:2015}. Since it was independently proposed by Scott and Montague in 1970~\cite{Scott:1970,Montague:1970}, neighborhood semantics has been a standard semantical tool for handling non-normal modal logics~\cite{Chellas:1980}.

A neighborhood semantics of contingency logic is proposed in~\cite{FanvD:neighborhood}. According to the interpretation, a formula $\phi$ is noncontingent, if and only if the proposition expressed by $\phi$ is a neighborhood of the evaluated state, or the complement of the proposition expressed by $\phi$ is a neighborhood of the evaluated state. This interpretation is in line with the philosophical intuition of noncontingency, viz. necessarily true or necessarily false. It is shown that contingency logic is less expressive than standard modal logic over various neighborhood model classes, and many neighborhood frame properties are undefinable in contingency logic. This brings about the difficulties in axiomatizing this logic over various neighborhood frames.

To our knowledge, only the classical contingency logic, i.e. the minimal system of contingency logic under neighborhood semantics, is presented in the literature~\cite{FanvD:neighborhood}. It is left as two open questions in~\cite{Bakhtiarietal:2017} what the axiomatizations of monotone contingency
logic and regular contingency logic are. In this paper, we will answer these two questions.

Besides, we also propose other proof systems up to the minimal contingency logic, and show their completeness with respect to the corresponding neighborhood frames. This will give a complete diagram which includes 8 systems, as~\cite[Fig.~8.1]{Chellas:1980} did for standard modal logic.

%Normally, for a logical system and a frame class, we can ask the following two related questions. First, given a certain class of frames, which logical system characterizes it, i.e. which logical system is sound and strongly complete with respect to the class in question? Second, given a logical system, which frame class is characterized by the system in question?

The remainder of this note is structured as follows. Section~\ref{sec.preliminaries} introduces some basics of contingency logic, such as its language, neighborhood semantics, axioms and rules. Sections~\ref{sec.sys-no-m} and~\ref{sec.sys-m} deal with the completeness of proof systems mentioned in Sec.~\ref{sec.preliminaries}, with or without a special axiom. The completeness proofs rely on the use of canonical neighborhood functions. In Sec.~\ref{sec.sys-no-m}, a simple canonical function is needed, while in Sec.~\ref{sec.sys-m} we need a more complex canonical function, which is inspired by a crucial function $\lambda$ used in a Kripke completeness proof in the literature. We further reflect on this $\lambda$ in Section~\ref{sec.reflection-lambda}, and show it is in fact equal to a related but complicated function originally given by Humberstone. We conclude with some discussions in Section~\ref{sec.concl}.

\section{Preliminaries}\label{sec.preliminaries}

Throughout this note, we fix $\BP$ to be a denumerable set of propositional variables. The language $\mathcal{L}(\Delta)$ of contingency logic is defined recursively as follows:
$$\phi::=p\in\BP\mid \neg\phi\mid \phi\land\phi\mid \Delta\phi.$$
$\Delta\phi$ is read ``it is non-contingent that $\phi$''. The contingency operator $\nabla$ abbreviates $\neg\Delta$. It does not matter which one of $\Delta$ and $\nabla$ is taken as primitive.

The neighborhood semantics of $\mathcal{L}(\Delta)$ is interpreted on neighborhood models. To say a triple $\M=\lr{S,N,V}$ is a neighborhood model, if $S$ is a nonempty set of states, $N:S\to\mathcal{P}(\mathcal{P}(S))$ is a neighborhood function, and $V$ is a valuation.

The following list of neighborhood properties is taken from~\cite[Def.~3]{FanvD:neighborhood}.
\begin{definition}[Neighborhood properties]\label{def.properties}\

$(n)$: $N(s)$ \emph{contains the unit}, if $S\in N(s)$. % Sometimes we say $N(s)$\emph{ satisfies condition (n)}.

$(i)$: $N(s)$ \emph{is closed under intersections}, if $X,Y\in N(s)$ implies $X\cap Y\in N(s)$. % Sometimes we say $N(s)$\emph{ satisfies condition (i)}.

$(s)$: $N(s)$ is \emph{supplemented}, or \emph{closed under supersets}, if $X\in N(s)$ and $X\subseteq Y\subseteq S$ implies $Y\in N(s)$. %We also call this property `monotonicity'. % Sometimes we say $N(s)$\emph{ satisfies condition (s)}.

$(c)$: $N(s)$ is\emph{ closed under complements}, if $X\in N(s)$ implies $S\backslash X\in N(s)$. % Sometimes we say $N(s)$\emph{ satisfies condition (c)}.
\end{definition}
Frame $\mathcal{F}=\lr{S,N}$ (and the corresponding model) possesses such a property P, if $N(s)$ has the property P for each $s\in S$, and we call the frame (resp. the model) P-frame (resp. P-model). Especially, a frame is called {\em quasi-filter}, if it possesses $(i)$ and $(s)$; a frame is called {\em filter}, if it has also $(n)$.

\medskip

Given a neighborhood model $\M=\lr{S,N,V}$ and a state $s\in S$, the semantics of $\phi\in\mathcal{L}(\Delta)$ is defined as follows~\cite{FanvD:neighborhood}, where $\phi^\M=\{s\in S\mid \M,s\vDash\phi\}$ is the truth set of $\phi$ (i.e. the proposition expressed by $\phi$) in $\M$.
\[
\begin{array}{|lll|}
\hline
\M,s\vDash p&\iff & s\in V(p)\\
\M,s\vDash \neg\phi&\iff &\M,s\nvDash \phi\\
\M,s\vDash\phi\land\psi&\iff &\M,s\vDash\phi\text{ and }\M,s\vDash\psi\\
\M,s\vDash\Delta\phi&\iff &\phi^\M\in N(s)\text{ or }S\backslash\phi^\M\in N(s)\\
\hline
\end{array}
\]

%Various (equivalent) proof systems have been proposed to characterize $\mathcal{L}(\Delta)$. For our purpose, we choose the following system.
Our discussions will be based on the following axioms and rules.
\[
\begin{array}{ll}
\text{TAUT}& \text{all instances of tautologies}\\
\Delta\text{Equ}&\Delta\phi\lra\Delta\neg\phi\\
\Delta\text{M}&\Delta\phi\to\Delta(\phi\vee\psi)\vee\Delta(\neg\phi\vee\chi)\\
\Delta\text{C}&\Delta\phi\land\Delta\psi\to\Delta(\phi\land\psi)\\
%\text{RN}\Delta&\dfrac{\phi}{\Delta\phi}\\
\Delta\text{N}&\Delta\top\\
\text{RE}\Delta&\dfrac{\phi\lra\psi}{\Delta\phi\lra\Delta\psi}\\
\end{array}
\]

\weg{We will define a sequence of systems as follows:
\[
\begin{array}{lll}
{\bf E^\Delta}=\text{TAUT}+\text{Equ}+\text{RE}&&{\bf M^\Delta}={\bf E^\Delta}+\Delta\text{M}\\
{\bf (EC)^\Delta}={\bf E^\Delta}+\Delta\text{C}&&{\bf (EN)^\Delta}={\bf E^\Delta}+\Delta\text{N}\\
{\bf R^\Delta}={\bf M^\Delta}+\Delta\text{C}&&{\bf (EMN)^\Delta}={\bf M^\Delta}+\Delta\text{N}\\
{\bf (ECN)^\Delta}={\bf (EC)^\Delta}+\Delta\text{N}&&{\bf K^\Delta}={\bf R^\Delta}+\Delta\text{N}\\
\end{array}
\]}

We will show that the following systems are sound and strongly complete with respect to the class of their corresponding frame classes.
\[
\begin{array}{|l|c|}
\hline
\text{systems} & \text{frame classes} \\
\hline
{\bf E^\Delta}=\text{TAUT}+\Delta\text{Equ}+\text{RE}\Delta&\text{all}\\
{\bf M^\Delta}={\bf E^\Delta}+\Delta\text{M}&(s)\\
{\bf (EC)^\Delta}={\bf E^\Delta}+\Delta\text{C}&(i)\&(c)\\
{\bf (EN)^\Delta}={\bf E^\Delta}+\Delta\text{N}&(n)\\
{\bf R^\Delta}={\bf M^\Delta}+\Delta\text{C}&\text{quasi-filters}\\
{\bf (EMN)^\Delta}={\bf M^\Delta}+\Delta\text{N}&(s)\& (n)\\
{\bf (ECN)^\Delta}={\bf (EC)^\Delta}+\Delta\text{N}&(i)\&(c)\&(n)\\
{\bf K^\Delta}={\bf R^\Delta}+\Delta\text{N}&\text{filters}\\
\hline
\end{array}
\]

\weg{\begin{proposition}
$\Box\phi\lra\bigwedge_{\psi\in\mathcal{L}(\Delta)}\Delta(\phi\vee\psi)$ is valid on $(s)$-frames.
\end{proposition}

\begin{proof}
Let $\M=\lr{S,N,V}$ be a $(s)$-model. Suppose $\M,s\vDash\Box\phi$, then $\phi^\M\in N(s)$. By $(s)$, for all $\psi\in\mathcal{L}(\Delta)$, we have $\phi^\M\cup\psi^\M\in N(s)$, i.e. $(\phi\vee\psi)^\M\in N(s)$, thus $\M,s\vDash\Delta(\phi\vee\psi)$.

Conversely, assume that $\M,s\vDash\Delta(\phi\vee\psi)$ for all $\psi\in\mathcal{L}(\Delta)$, then $(\phi\vee\psi)^\M\in N(s)$ or $(\neg (\phi\vee\psi))^\M\in N(s)$ for all $\psi\in\mathcal{L}(\Delta)$.
\end{proof}}

Given a system $\Lambda$ and a maximal consistent set $S^c$ for $\Lambda$, let $|\phi|_\Lambda$ be the proof set of $\phi$ in $\Lambda$; in symbol, $|\phi|_\Lambda=\{s\in S^c\mid \phi\in s\}$. It is easy to show that $|\neg\phi|_\Lambda=S^c\backslash |\phi|_\Lambda$. We always omit the subscript $\Lambda$ when it is clear from the context.

\section{Systems excluding $\Delta$M}\label{sec.sys-no-m}

Given a proof system, a standard method of showing its completeness under neighborhood semantics is constructing the canonical neighborhood model, where one essential part is the definition of canonical neighborhood function.

%The following definition is taken from~\cite[Def.~9]{FanvD:neighborhood}.

\begin{definition}\label{def.cm-no-m} Let $\Sigma$ be a system excluding $\Delta$M.
A tuple $\M^c=\lr{S^c,N^c,V^c}$ is a {\em canonical neighborhood model} for $\Sigma$, if
\begin{itemize}
\item $S^c=\{s\mid s\text{ is a maximal consistent set for }\Sigma\}$,
\item $N^c(s)=\{|\phi|\mid \Delta\phi\in s\}$,
\item $V^c(p)=|p|$.
\end{itemize}
%where $|\phi|=\{s\in S^c\mid \phi\in s\}$ is the proof set of $\phi$ in ${\bf E^\Delta}$.
\end{definition}

\begin{theorem}\cite[Thm.~1]{FanvD:neighborhood}\label{thm.comp-E}
${\bf E^\Delta}$ is sound and strongly complete with respect to the class of all neighborhood frames.
\end{theorem}

In what follows, we will extend the canonical model construction to all systems excluding $\Delta$M listed above.

It is not hard to show that $\Delta$C is invalid on the class of all $(i)$-frames, and thus ${\bf (EC)^\Delta}$ is {\em not} sound (and strongly complete) with respect to the class of all frames satisfying $(i)$. Despite this, it is indeed sound and strongly complete with respect to a more restricted frame class.

\begin{theorem}\label{thm.comp-EC}
${\bf (EC)^\Delta}$ is sound and strongly complete with respect to the class of all frames satisfying $(i)\&(c)$.
\end{theorem}

\begin{proof}
By Thm.~\ref{thm.comp-E}, it suffices to show that $\Delta$C is valid on $(i)\&(c)$-frames, and that $N^c$ possesses $(i)$ and $(c)$. The former follows from the fact that $\Delta$C is valid on the class of $(i)\&(c)$-frames under a new semantics proposed in~\cite{Fan:2017} and that on $(c)$-frames, the current semantics and the new semantics satisfies the same formulas.

As for the latter, $\Delta$Equ guarantees $(c)$, and $\Delta$C provides $(i)$.
\end{proof}

\begin{theorem}\label{thm.comp-EN}
${\bf (EN)^\Delta}$ is sound and strongly complete with respect to the class of all frames satisfying $(n)$.
\end{theorem}

\begin{proof}
It suffices to show that $N^c$ possesses the property $(n)$. This is immediate due to $\Delta$N and the definition of $N^c$.
\end{proof}

By Thm.~\ref{thm.comp-EC} and Thm.~\ref{thm.comp-EN}, we obtain the following result.
\begin{theorem}
${\bf (ECN)^\Delta}$ is sound and strongly complete with respect to the class of all frames satisfying $(i)\&(c)\&(n)$.
\end{theorem}

\section{Systems including $\Delta$M}\label{sec.sys-m}

In this section, we show that the systems including $\Delta$M listed above are sound and strongly complete with respect to the corresponding frame classes.

We first consider the system ${\bf M^\Delta}$. The following result tells us that ${\bf M^\Delta}$ is sound with respect to the class of frames satisfying $(s)$.

%\section{Monotonic and Regular contingency logics}

%Give a sequent calculus that is sound and strongly complete with respect to the class of neighborhood frames under the old neighborhood semantics, and then translate the sequent calculus to an equivalent Hilbert-style system.

%$\mathbb{M}^\Delta\vdash\Delta\phi\to\Delta\top$ this is equivalent to $\dfrac{\psi}{\Delta\phi\to\Delta\psi}$

%$\mathbb{M}^\Delta\vdash\Delta\phi\to\Delta(\phi\to\psi)\vee\Delta(\neg\phi\to\chi)$

%$\mathbb{R}^\Delta\vdash\Delta\phi\land\Delta\psi\to\Delta(\phi\land\psi)$

\begin{proposition}
$\Delta$M is valid on the class of frames satisfying $(s)$.
\end{proposition}

\begin{proof}
Let $\M=\lr{S,N,V}$ be a $(s)$-model and $s\in S$. Suppose that $\M,s\vDash\Delta\phi$, then $\phi^\M\in N(s)$ or $(\neg\phi)^\M\in N(s)$. If $\phi^\M\in N(s)$, then by $(s)$, $\phi^\M\cup\psi^\M\in N(s)$, which implies $\M,s\vDash\Delta(\phi\vee\psi)$; if $(\neg\phi)^\M\in N(s)$, then similarly, we can obtain $\M,s\vDash\Delta (\neg\phi\vee\chi)$. Either case gives us $\M,s\vDash\Delta (\phi\vee\psi)\vee\Delta(\neg\phi\vee\chi)$, as required.
\end{proof}

For the completeness, we construct the canonical neighborhood model for ${\bf M^\Delta}$, where the crucial definition is the canonical neighborhood function. The definition of $N^c$ below is inspired by a function $\lambda$ introduced in~\cite{DBLP:journals/ndjfl/Kuhn95}.\footnote{The difference between $N^c$ and $\lambda$ lies in the codomains: $N^c$'s codomain is $\mathcal{P}(\mathcal{P}(S^c))$, whereas $\lambda$'s is $\mathcal{P}(\mathcal{L}(\Delta))$.}

\begin{definition}\label{def.canonicalmodel} Let $\Gamma$ be a system including $\Delta$M.
A triple $\M^c=\lr{S^c,N^c,V^c}$ is a {\em canonical model} for $\Gamma$, if
\begin{itemize}
\item $S^c=\{s\mid s\text{ is a maximal consistent set for }\Gamma\}$,
\item For each $s\in S^c$, $N^c(s)=\{|\phi|\mid \Delta(\phi\vee\psi)\in s\text{ for every }\psi\}$,
\item For each $p\in \BP$, $V^c(p)=|p|$.
\end{itemize}
\end{definition}

We need to show that $N^c$ is well-defined.
\begin{lemma}
If $|\phi|=|\psi|$, then $|\phi|\in N^c(s)$ iff $|\psi|\in N^c(s)$.
\end{lemma}

\begin{proof}
Suppose that $|\phi|=|\psi|$, then $\vdash\phi\lra\psi$, then for every $\chi$, $\vdash\phi\vee\chi\lra\psi\vee\chi$. By RE$\Delta$, we have $\vdash\Delta(\phi\vee\chi)\lra\Delta(\psi\vee\chi)$, thus for every $\chi$, $\Delta(\phi\vee\chi)\in s$ iff for every $\chi$, $\Delta(\psi\vee\chi)\in s$, and therefore $|\phi|\in N^c(s)$ iff $|\psi|\in N^c(s)$.
\end{proof}

\begin{lemma}\label{lemma.truthlemma-mc} Let $\M^c$ be a canonical model for ${\bf M^\Delta}$. Then for all $\phi\in \mathcal{L}(\Delta)$, for all $s\in S^c$, we have
$\M^c,s\vDash\phi\iff \phi\in s,$ i.e. $\phi^{\M^c}=|\phi|$.
\end{lemma}

\begin{proof}
By induction on $\phi$. The only nontrivial case is $\Delta\phi$.

%Suppose that $\Delta\phi\in s$, to show that $\M^c,s\vDash\Delta\phi$. By supposition and axiom $\text{Dis}\Delta$, $\Delta(\phi\vee\psi)\vee\Delta(\neg\phi\vee\chi)\in s$ for any $\psi$ and $\chi$. Thus $\Delta(\phi\vee\psi)\in s$ for any $\psi$, or $\Delta(\neg\phi\vee\chi)\in s$ for any $\chi$. The first case implies $|\phi|\in N^c(s)$, by induction hypothesis, $\phi^{\M^c}\in N^c(s)$; the second case implies $|\neg\phi|\in N^c(s)$, i.e. $S^c\backslash |\phi|\in N^c(s)$, then by induction hypothesis, $(\neg\phi)^{\M^c}=S^c\backslash \phi^{\M^c}\in N^c(s)$. Both entails $\M^c,s\vDash\Delta\phi$.

Suppose, for a contradiction, that $\Delta\phi\in s$ but $\M^c,s\nvDash\Delta\phi$. Then by induction hypothesis, we obtain $|\phi|\notin N^c(s)$, and $S^c\backslash|\phi|\notin N^c(s)$, i.e. $|\neg\phi|\notin N^c(s)$. Thus $\Delta(\phi\vee\psi)\notin s$ for some $\psi$, and $\Delta(\neg\phi\vee\chi)\notin s$ for some $\chi$. Using axiom $\Delta$M, we obtain $\Delta\phi\notin s$: a contradiction.

Conversely, assume that $\M^c,s\vDash\Delta\phi$, to show that $\Delta\phi\in s$. By assumption and induction hypothesis, we have $|\phi|\in N^c(s)$, or $S^c\backslash|\phi|\in N^c(s)$, i.e. $|\neg\phi|\in N^c(s)$. If $|\phi|\in N^c(s)$, then for every $\psi$, $\Delta(\phi\vee\psi)\in s$. In particular, $\Delta\phi\in s$; if $|\neg\phi|\in N^c(s)$, then by a similar argument, we obtain $\Delta\neg\phi\in s$, thus $\Delta\phi\in s$. Therefore, $\Delta\phi\in s$.
\end{proof}

Note that $\M^c$ is not necessarily supplemented. Thus we need to define a notion of supplementation, which comes from~\cite{Chellas:1980}.

\begin{definition}
Let $\M=\lr{S,N,V}$ be a neighborhood model. The {\em supplementation} of $\M$, denoted $\M^+$, is a triple $\lr{S,N^+,V}$, where for each $s\in S$, $N^+(s)$ is the superset closure of $N(s)$, i.e. for every $s\in S$ and $X\subseteq S$,
$$N^+(s)=\{X\mid Y\subseteq X\text{ for some }Y\in N(s)\}.$$
\end{definition}

It is easy to see that $\M^+$ is supplemented. Moreover, $N(s)\subseteq N^+(s)$. The proof below is a routine work.

\begin{proposition}\label{prop.i&n}
Let $\M$ be a neighborhood model. If $\M$ possesses the property $(i)$, then so does $\M^+$; if $\M$ possesses the property $(n)$, then so does $\M^+$.
\end{proposition}

We will denote the supplementation of $\M^c$ by $(\M^c)^+=\lr{S^c,(N^c)^+,V^c}$. To demonstrate the completeness of ${\bf M^\Delta}$ with respect to the class of $(s)$-frames, we need only show that $(\M^c)^+$ is a canonical model for ${\bf M^\Delta}$. That is,
\begin{lemma}\label{lem.canonicalmodel} For each $s\in S^c$,
$$|\phi|\in (N^c)^+(s)\iff \Delta(\phi\vee\psi)\in s\text{ for every }\psi.$$
\end{lemma}

\begin{proof}
`$\Longleftarrow$': immediate by $N^c(s)\subseteq (N^c)^+(s)$ for each $s\in S^c$ and the definition of $N^c$.

`$\Longrightarrow$': Suppose that $|\phi|\in (N^c)^+(s)$, to show that $\Delta(\phi\vee\psi)\in s\text{ for every }\psi$. By supposition, $X\subseteq |\phi|$ for some $X\in N^c(s)$. Then there is a $\chi$ such that $|\chi|=X$, and thus $\Delta(\chi\vee\psi)\in s$ for every $\psi$, in particular $\Delta(\chi\vee\phi\vee\psi)\in s$. From $|\chi|\subseteq |\phi|$ follows that $\vdash\chi\to\phi$, thus $\vdash\chi\vee\phi\vee\psi\lra\phi\vee\psi$, and hence $\vdash\Delta(\chi\vee\phi\vee\psi)\lra\Delta(\phi\vee\psi)$ by RE$\Delta$. Therefore $\Delta(\phi\vee\psi)\in s$ for every $\psi$.
\end{proof}

\weg{\begin{lemma}
For all $\phi\in \mathcal{L}(\Delta)$, for all $s\in S^c$, we have
$\M^+,s\vDash\phi\iff \phi\in s,$ i.e. $\phi^{\M^+}=|\phi|$.
\end{lemma}

\begin{proof}
By induction on $\phi$. The nontrivial case is $\Delta\phi$.

Suppose that $\Delta\phi\in s$, to show that $\M^+,s\vDash\Delta\phi$. By supposition and axiom $\text{Dis}\Delta$, $\Delta(\phi\vee\psi)\vee\Delta(\neg\phi\vee\chi)\in s$ for arbitrary $\psi,\chi$. Thus $\Delta(\phi\vee\psi)\in s$ for all $\psi$, or $\Delta(\neg\phi\vee\chi)\in s$ for all $\chi$. The first case implies $|\phi|\in N^c(s)$, by induction hypothesis, $\phi^{\M^+}\in N^c(s)$, thus $\phi^{\M^+}\in N^+(s)$; similarly, the second case implies $(\neg\phi)^{\M^+}\in N^+(s)$. Both entails $\M^+,s\vDash\Delta\phi$.

Conversely, assume that $\M^+,s\vDash\Delta\phi$, then by induction hypothesis, $|\phi|\in N^+(s)$ or $|\neg\phi|\in N^+(s)$. If $|\phi|\in N^+(s)$, then $X\subseteq |\phi|$ for some $X\in N^c(s)$, and hence for some $\psi$, $|\psi|=X\in N^c(s)$, i.e. $\Delta(\psi\vee\psi')\in s$ for all $\psi'$. By $|\psi|\subseteq \phi$, we have $\vdash\psi\to\phi$, then $\vdash\psi\vee\phi\lra\phi$, by RE$\Delta$, $\vdash\Delta(\psi\vee\phi)\lra\Delta\phi$. Since $\Delta(\psi\vee\phi)\in s$, we conclude that $\Delta\phi\in s$. Similarly, from $|\neg\phi|\in N^+(s)$, it can follow that $\Delta\neg\phi\in s$, and therefore $\Delta\phi\in s$, as desired.
\end{proof}}

By Lemma~\ref{lemma.truthlemma-mc} and Lemma~\ref{lem.canonicalmodel}, we have

\begin{lemma}
For all $\phi\in \mathcal{L}(\Delta)$, for all $s\in S^c$, we have
$(\M^c)^+,s\vDash\phi\iff \phi\in s,$ i.e. $\phi^{(\M^c)^+}=|\phi|$.
\end{lemma}

With a routine work, we obtain
\begin{theorem}\label{thm.comp-m}
${\bf M^\Delta}$ is sound and strongly complete with respect to the class of frames satisfying $(s)$.
\end{theorem}

We are now in a position to deal with the sound and strong completeness of ${\bf R^\Delta}$. First, the soundness follows from the following result.
%However, it is indeed valid on the class of quasi-filters.
\begin{proposition}
$\Delta$C is valid on the class of quasi-filters.
\end{proposition}

\begin{proof}
Let $\M=\lr{S,N,V}$ be a $(s)$-model and $s\in S$. Suppose that $\M,s\vDash\Delta\phi\land\Delta\psi$, then $\phi^\M\in N(s)$ or $(\neg\phi)^\M\in N(s)$, and $\psi^\M\in N(s)$ or $(\neg\psi)^\M\in N(s)$. Consider the following three cases:
\begin{itemize}
\item $\phi^\M\in N(s)$ and $\psi^\M\in N(s)$. By $(i)$, we obtain $\phi^\M\cap \psi^\M\in N(s)$, i.e. $(\phi\land\psi)^\M\in N(s)$, which gives $\M,s\vDash\Delta(\phi\land\psi)$.
\item $(\neg\phi)^\M\in N(s)$. By $(s)$, we infer $(\neg\phi)^\M\cup(\neg\psi)^\M\in N(s)$, i.e. $(\neg(\phi\land\psi))^\M\in N(s)$, which implies $\M,s\vDash\Delta(\phi\land\psi)$.
\item $(\neg\psi)^\M\in N(s)$. Similar to the second case, we can derive that $\M,s\vDash\Delta(\phi\land\psi)$.
\end{itemize}
\end{proof}

\begin{proposition}\label{prop.i} Let $\M^c$ be a canonical model for ${\bf R^\Delta}$. Then $\M^c$ possesses the property $(i)$. As a corollary, $(\M^c)^+$ also possesses the property $(i)$.
\end{proposition}

\begin{proof}
Suppose $X\in N^c(s)$ and $Y\in N^c(s)$, to show that $X\cap Y\in N^c(s)$. By supposition, there exist $\phi$ and $\chi$ such that $X=|\phi|$ and $Y=|\chi|$, and then $\Delta(\phi\vee\psi)\in s$ for every $\psi$, and $\Delta(\chi\vee\psi)\in s$ for every $\psi$. Using axiom $\Delta$C, we infer $\Delta((\phi\land\chi)\vee\psi)\in s$ for every $\psi$. Therefore, $|\phi\land\chi|\in N^c(s)$, i.e. $X\cap Y\in N^c(s)$. Then it follows that $(\M^c)^+$ also possesses the property $(i)$ from Prop.~\ref{prop.i&n}.
\end{proof}

The results below are now immediate, due to Thm.~\ref{thm.comp-m} and Prop.~\ref{prop.i}.
\begin{theorem}\label{thm.comp-r}
${\bf R^\Delta}$ is sound and strongly complete with respect to the class of quasi-filters.
\end{theorem}

\begin{proposition}\label{prop.n-pres}Let $\M^c$ be a canonical model for ${\bf (EMN)^\Delta}$. Then $\M^c$ possesses the property $(n)$. As a corollary, $(\M^c)^+$ also possesses $(n)$.
\end{proposition}

\begin{theorem}
${\bf (EMN)^\Delta}$ is sound and strongly complete with respect to the class of frames satisfying $(s)\&(n)$.
\end{theorem}

%The result below follows immediately from Thm.~\ref{thm.comp-r} and Prop.~\ref{prop.n-pres}.
It is straightforward to obtain the following result.
\begin{theorem}
${\bf K^\Delta}$ is sound and strongly complete with respect to the class of filters.
\end{theorem}

By constructing countermodels, we can obtain the following cube, which summarizes the deductive powers of the systems in this paper. An arrow from a system $S_1$ to another $S_2$ means that $S_2$ is deductively stronger than $S_1$.
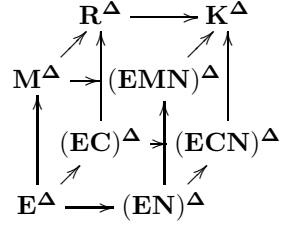
\[\xymatrix@!0{
  & {\bf R^\Delta} \ar[rr]
      &  & {\bf K^\Delta}        \\
  {\bf M^\Delta} \ar[ur]\ar[rr]
      &  & {\bf (EMN)^\Delta} \ar[ur] \\
  & {\bf (EC)^\Delta}\ar[uu]\ar[rr]
      &  & {\bf (ECN)^\Delta}   \ar[uu]            \\
   {\bf E^\Delta} \ar[rr]\ar[ur]\ar[uu]
     &  & {\bf (EN)^\Delta} \ar[ur] \ar[uu]       }\]

\section{Reflection: how does the function $\lambda$ arise?}\label{sec.reflection-lambda}

As noted, in order to show the completeness of proof systems including $\Delta$M, a crucial part is to define a suitable canonical function, i.e. $N^c$, which is inspired by the function $\lambda$ in~\cite{DBLP:journals/ndjfl/Kuhn95}. The $\lambda$ is very important for the definition of canonical relation and thus for the completeness proof in the cited paper. It is this function that helps find simple axiomatizations for the minimal contingency logic and transitive contingency logic under Kripke semantics, so to speak. Despite its importance, the author did not say any intuitive idea about $\lambda$. And this function was thought of as `ingenious' creation by some other researchers, say Humberstone~\cite[p.~118]{Humberstone:2002} and Fan, Wang and van Ditmarch~\cite[p.~101]{Fanetal:2015}. But how does the function arise? In this section, we unfold the mystery of $\lambda$, and show that it is actually equal to a related function $\lambda$ proposed in Humberstone~\cite{Humberstone95}.

%An axiomatization was proposed in~\cite{Humberstone95}, with a complicated proof method which needs to resort to K\"{o}nig's Lemma.
To show completeness of minimal non-contingency logic under Kripke semantics, Humberstone~\cite[p.~219]{Humberstone95} defined the canonical relation $R^c$ as $xR^cy$ iff $\lambda(x)\subseteq y$, where, denoted by H's $\lambda$,
$$\lambda(x)=\{\phi\mid \Delta \phi\in x\text{ and }\forall \psi\text{ such that }\vdash \phi\to \psi,\Delta \psi\in x\}.\footnote{In the definition of $\lambda$, Humberstone used $A$ and $B$ rather than $\phi$ and $\psi$, respectively. To maintain the consistency of notation in this paper, we here use $\phi$ and $\psi$ instead.}$$
The reason for defining the function $\lambda$ in such a way, is that the author would like to `simulate' the canonical relation of the minimal modal logic, which is defined via $xRy$ iff $\lambda(x)\subseteq y$, where $\lambda(x)=\{\phi\mid \Box \phi\in x\}$. This can be seen from several passages:

\medskip

{\em The intuitive idea is that for $x\in W$, $\lambda(x)$ is the set of formulas which are necessary at $x$.}

$\cdots$

{\em The idea of the entry condition on $A$, that only such $A$ (with $\Delta A\in x$) should be labeled as Necessary if all their consequences are non-contingent, is that $\cdots$, those non-contingencies which qualify as such because they, rather than their negations, are necessary and have only non-contingent consequences, since those consequences are themselves necessary.}~\cite[p.~219]{Humberstone95}

\medskip

Then the function $\lambda$ was simplified, and accordingly, the completeness proof was simplified in~\cite{DBLP:journals/ndjfl/Kuhn95}. There, $\lambda(x)$, denoted by K's $\lambda$, is defined as:
$$\lambda(x)=\{\phi\mid\forall \psi, \Delta(\phi\vee \psi)\in x\}.$$
%But the author of~\cite{DBLP:journals/ndjfl/Kuhn95} did not say any intuitive idea of the simplified function $\lambda$. However, as we will demonstrate, in fact, Kuhn's $\lambda$ is equal to Humberstone's $\lambda$.
In the sequel, we will demonstrate that, in fact, K's $\lambda$ is equal to H's $\lambda$.

To begin with, notice that $\vdash \phi\to \phi$, thus the part following `and' in the H's $\lambda$ definition entails $\Delta \phi\in x$. Therefore, the H's $\lambda(x)$ is equal to a simplified version:
$$\lambda(x)=\{\phi\mid \forall \psi\text{ such that }\vdash \phi\to \psi,\Delta \psi\in x\}.$$
Then it is sufficient to show that the simplified $\lambda$ is further equal to K's $\lambda$, even in the setting of arbitrary neighborhood contingency logics (as opposed to Kripke contingency logics).
\begin{proposition} Let $x$ be a maximal consistent set. Given the rule RE$\Delta$, the following statements are equivalent.\footnote{RE$\Delta$ is just ($\Delta$Cong) in~\cite{Humberstone95}.}

(1) For every $\psi\text{ such that }\vdash \phi\to \psi$, $\Delta \psi\in x$.

(2) For every $\psi$, $\Delta(\phi\vee \psi)\in x$.
\end{proposition}

\begin{proof}
$(1)\Longrightarrow(2)$: suppose (1) holds. Since $\vdash\phi\to\phi\vee\psi$, then it is immediate by (1) that $\Delta(\phi\vee\psi)\in x$, namely (2).

$(2)\Longrightarrow(1)$: suppose (2) holds, to show (1). For this, assume that $\vdash\phi\to\psi$, then $\vdash\phi\vee\psi\lra\psi$, by RE$\Delta$, $\vdash\Delta(\phi\vee\psi)\lra\Delta\psi$. By (2), we obtain that $\Delta\psi\in x$, as desired.
%This is because the condition that (i) ``$\forall B$ such that $\vdash A\to B$, $\Delta B\in x$'' is equivalent to the condition that (ii) ``$\forall B, \Delta(A\vee B)\in x$'', given the axiom  inference rule $\dfrac{A}{\Delta A}$ (this is (1.6) in~\cite[p.~216]{Humberstone95}). We prove the equivalence as follows. Let the inference rule in question be given. Firstly, suppose (i) holds, since for all $B$, $\vdash A\to A\vee B$, (i) gives $\Delta(A\vee B)\in x$, viz. (ii). Conversely, suppose (ii) holds, and $\vdash A\to B$, by the inference rule, $\vdash\Delta(A\to B)$, and thus $\Delta(A\to B)\in x$; moreover, by (ii) and $\Delta(A\vee B)\lra \Delta(\neg A\to B)$, we have $\Delta(\neg A\to B)\in x$. Now combining $\Delta(A\to B)\in x$ and $\Delta(\neg A\to B)\in x$, and using axiom $\Delta A\land\Delta B\to\Delta (A\land B)$ (this is a special case of (2.2) in~\cite[p.~217]{Humberstone95}), we obtain that $\Delta B\in x$. Since for all $B$, from the supposition (ii) and $\vdash A\to B$, it follows that $\Delta B\in x$. Thus we conclude that (ii) implies (i).
\end{proof}

%\section{Conclusion and Discussions}\label{sec.concl}
\section{Concluding Discussions}\label{sec.concl}

In this note, by defining suitable neighborhood canonical functions, we presented a sequence of contingency logics under neighborhood semantics. In particular, inspired by Kuhn's function $\lambda$ in~\cite{DBLP:journals/ndjfl/Kuhn95}, we defined a desired canonical neighborhood function, and then axiomatized monotone contingency logic and regular contingency logic and other logics including the axiom $\Delta$M, thereby answering two open questions raised in~\cite{Bakhtiarietal:2017}. We then reflected on the function $\lambda$, and showed that it is actually equal to Humberstone's function $\lambda$ in~\cite{Humberstone95}, even in the setting of arbitrary neighborhood contingency logics.

Moreover, as we observe, in ${\bf M^\Delta}$, $\Delta$M can be replaced by $\Delta\phi\to\Delta(\phi\to\psi)\vee \Delta(\neg\phi\to\chi)$, and in ${\bf R^\Delta}$, $\Delta$C can be replaced by $\Delta(\psi\to\phi)\land\Delta(\neg\psi\to\phi)\to\Delta\phi$.\footnote{The hard part is the direction from $\Delta(\psi\to\phi)\land\Delta(\neg\psi\to\phi)\to\Delta\phi$ to $\Delta$C. The proof details for this, we refer to~\cite[Prop.~50]{DBLP:journals/corr/FanWD13}, where knowing whether operator $\textit{Kw}$ is the epistemic reading of $\Delta$.} Thus we can also adopt these two alternative formulas to axiomatize monotone contingency logic and regular contingency logic. Therefore, it was {\em wrong} to claim that ``one cannot fill these gaps with the axioms $\Delta\phi\to\Delta(\phi\to\psi)\vee \Delta(\neg\phi\to\chi)$ and $\Delta(\psi\to\phi)\land\Delta(\neg\psi\to\phi)\to\Delta\phi$'' on~\cite[p.~62]{Bakhtiarietal:2017}.

Recall that an `almost definability' schema, $\nabla\chi\to(\Box\phi\lra(\Delta\phi\land\Delta(\chi\to\phi)))$, is proposed in~\cite{Fanetal:2014}, and shown in~\cite{Fanetal:2015} to be applied to axiomatize contingency logic over much more Kripke frame classes than Kuhn's function $\lambda$ and other variations. Therefore, it may be natural to ask if the schema can also work in the neighborhood setting. The canonical neighborhood function inspired by the schema seems to be
$$N(s)=\{|\phi|\mid \Delta\phi\land\Delta(\psi\to\phi)\in s\text{ for some }\nabla\psi\in s\}.$$
Unfortunately the answer seems to be negative. The reason can be explained as follows. Although $N^c$ in Def.~\ref{def.canonicalmodel} is almost monotonic in the sense that if $|\phi|\in N^c(s)$ and $|\phi|\subseteq |\psi|$, then $|\psi|\in N^c(s)$, as can be easily seen from the proof of Lemma~\ref{lem.canonicalmodel}, in contrast, as one may easily verify, $N$ is not almost monotonic in the above sense, i.e., it fails that if $|\phi|\in N(s)$ and $|\phi|\subseteq |\psi|$, then $|\psi|\in N(s)$.
%Recall that $N^c$ is inspired by the Kuhn's function $\lambda$.
%$$N^c(s)=\{|\phi|\mid \Delta(\phi\vee\psi)\in s\text{ for every }\psi\}$$
This can also explain why $N^c$ works well for monotone and regular contingency logics and other logics including the axiom $\Delta$M. Despite this fact, this $N^c$ does not apply to systems excluding the axiom $\Delta$M, since we need this axiom to ensure the truth lemma (Lemma~\ref{lemma.truthlemma-mc}). It is also worth noting that this $N^c$ is smaller than that in the case of classical contingency logic (Def.~\ref{def.cm-no-m}), thus we cannot address all neighborhood contingency logics in a unified way.\footnote{In contrast, the canonical neighborhood function used in the completeness proof of classical modal logic is the smallest neighborhood function among canonical neighborhood functions used in the completeness proofs of all neighborhood modal logics. Cf. e.g.~\cite{Chellas:1980}.} This indicates that the completeness proofs of these logics are nontrivial. Besides, $N^c$ seems not workable for proper extensions of ${\bf K^\Delta}$, which we leave for future work.

\section{Acknowledgements}

This research is supported by the youth project 17CZX053 of National Social Science Fundation of China. The author would like to thank Lloyd Humberstone for careful reading of the manuscript and insightful comments.

\bibliographystyle{plain}
\bibliography{biblio2017,biblio2016}

\end{document}